\documentclass[a4paper]{amsart}
\usepackage{wrapfig}
\usepackage[dvips]{graphicx}
\usepackage{amsmath,amsthm,amssymb,amscd}
\theoremstyle{definition}
\newtheorem{thm}{Theorem}[section]
\newtheorem{Def}[thm]{Definition}
\newtheorem{pro}[thm]{Proposition}
\newtheorem{cor}[thm]{Corollary}
\newtheorem{lem}[thm]{Lemma}
\newtheorem{ex}[thm]{Example}
\newtheorem{rem}[thm]{Remark}
\theoremstyle{definition}

\begin{document}

\title{Picard groups of certain stably projectionless C$^*$-algebras}
\author{Norio Nawata}
\address{Department of Mathematics and Informatics, 
Graduate school of Science,
Chiba University,1-33 Yayoi-cho, Inage, Chiba, 263-8522, Japan}
\email{nawata@math.s.chiba-u.ac.jp}
\keywords{Picard group; Fundamental group; Stably projectionless C$^*$-algebra; Cuntz semigroup; Kirchberg's central sequence algebra}
\subjclass[2010]{Primary 46L05, Secondary 46L08; 46L35}
\thanks{The author is a Research Fellow of the Japan Society for the Promotion of Science.}
\begin{abstract} 
We compute Picard groups of several nuclear and non-nuclear 
simple stably projectionless C$^*$-algebras. 
In particular, the Picard group of the Razak-Jacelon algebra $\mathcal{W}_{2}$ is isomorphic 
to a semidirect product  of $\mathrm{Out}(\mathcal{W}_{2})$ with $\mathbb{R}_{+}^\times$. 
Moreover, for any separable simple nuclear stably projectionless C$^*$-algebra with 
a finite dimensional lattice of densely defined lower semicontinuous traces, 
we show that $\mathcal{Z}$-stability and strict comparison 
are equivalent. (This is essentially based on the result of Matui and Sato, and Kirchberg's central sequence algebras.) 
This shows if $A$ is a separable simple nuclear stably projectionless  C$^*$-algebra with a unique tracial state 
(and no unbounded trace) and has strict comparison, the following sequence is exact: 
\[\begin{CD}
      {1} @>>> \mathrm{Out}(A) @>>> \mathrm{Pic}(A) @>>> \mathcal{F}(A)
 @>>> {1} \end{CD} \] 
where $\mathcal{F}(A)$ is the fundamental group of $A$. 
\end{abstract}
\maketitle

\section{Introduction} 
Let $A$ be a C$^*$-algebra. 
Brown, Green and Rieffel introduced the Picard group $\mathrm{Pic}(A)$ of $A$ in 
\cite{BGR}. We say that an automorphism $\alpha$ of $A$ is \textit{inner} 
if there exists a unitary element $u$ in the multiplier algebra $M(A)$ of $A$ 
such that $\alpha (a)=uau^*$ for any $a\in A$. 
Let $\mathrm{Inn}(A)$ denote the set of inner automorphisms of $A$, and let 
$\mathrm{Out}(A)=\mathrm{Aut}(A)/\mathrm{Inn}(A)$. 
They showed that if $A$ is $\sigma$-unital, then $\mathrm{Pic}(A)$ is isomorphic 
to $\mathrm{Out}(A\otimes\mathbb{K})$. 
Kodaka computed Picard groups of several unital C$^*$-algebras in \cite{kod1}, \cite{kod2} 
and \cite{kod3}. 
In particular he computed the Picard groups of the irrational rotation algebras $A_\theta$. 
If $\theta$ is not quadratic irrational number, then $\mathrm{Pic}(A)$ is isomorphic 
to $\mathrm{Out}(A_\theta)$ and if $\theta$ is a quadratic number, then 
$\mathrm{Pic}(A_\theta)$ is isomorphic to $\mathrm{Out}(A_\theta )\rtimes \mathbb{Z}$. 
Kodaka considered the following set 
$$
\mathrm{FP}/\sim =\{[p]\;|\;p\text{ is a full projection in }A\otimes\mathbb{K}\text{ such that }
p(A\otimes\mathbb{K})p\cong A \} 
$$
where $[p]$ is the Murray-von Neumann equivalence class of $p$ 
and showed that if $\mathrm{Out}(A)$ is a normal subgroup of 
$\mathrm{Out}(A\otimes\mathbb{K})$ and $A$ is unital, then $\mathrm{FP}/\sim$ has a suitable group structure 
and the following sequence is exact:
\[\begin{CD}
      {1} @>>> \mathrm{Out}(A) @>>> \mathrm{Pic}(A) @>>> \mathrm{FP}/\sim
 @>>> {1} \end{CD}. \] 
Note that there exists a simple unital AF algebra $B$ with a unique tracial state such that 
$\mathrm{FP}/\sim$ of $B$ does not have any suitable group structure. 
If A is unital, K-theoretical method enables us to show that $\mathrm{Out}(A)$ is a normal subgroup of 
$\mathrm{Out}(A\otimes\mathbb{K})$ (see \cite[Proposition 1.5]{kod1}). 

The set of $\mathrm{FP}/\sim$ is similar to the fundamental group 
$\mathcal{F}(M)$ of a II$_1$ factor $M$ introduced by Murray and von Neumann in \cite{MN}. 
Watatani and the author introduced the fundamental group $\mathcal{F}(A)$ 
of a simple unital C$^*$-algebra $A$ with a unique tracial state $\tau$ 
based on Kodaka's results. 
The fundamental group ${\mathcal F}(A)$ is defined as the set of the numbers 
$\tau \otimes \mathrm{Tr}(p)$ for some projection $p \in M_n(A)$ such that $pM_n(A)p$ is 
isomorphic to $A$. We showed that $\mathcal{F}(A)$ is a multiplicative 
subgroup of $\mathbb{R}_{+}^\times$ and computed fundamental groups of several 
C$^*$-algebras in \cite{NW}. 
Moreover we showed that any countable subgroup of $\mathbb{R}_+^\times$ 
can be realized as the fundamental group of a separable simple unital C$^*$-algebra 
with a unique tracial state in \cite{NW2}. Note that the fundamental groups of separable simple 
unital C$^*$-algebras are countable. 
Furthermore the author introduced the fundamental group of a simple  
stably projectionless C$^*$-algebra with unique (up to scalar multiple) densely defined 
lower semicontinuous trace $\tau$ in \cite{Na}. 
If $\tau$ is a tracial state and $A$ is $\sigma$-unital, then the fundamental group of 
${\mathcal F}(A)$ of $A$ is defined as the set of the numbers $d_{\tau}(h)$ for some 
positive element $h\in A\otimes\mathbb{K}$ such that $\overline{h(A\otimes\mathbb{K})h}$ 
is isomorphic to $A$ where $d_\tau$ is the dimension function defined by $\tau$. 
Note that if $A$ is unital, then this definition coincides 
with the previous definition and there exist separable simple stably projectionless 
C$^*$-algebras such that their fundamental groups are equal to $\mathbb{R}_+^\times$. 
The fundamental group of a II$_1$ factor $M$ is equal to the set of trace-scaling 
constants for automorphisms of a $\mathrm{II}_{\infty}$ factor  $M\otimes B(\mathcal{H})$. 
This characterization shows that the fundamental groups of II$_1$ factors are related to the 
structure theorem for type III$_\lambda$ factors where $0<\lambda \leq 1$ 
(see \cite{Tak1} and \cite{Tak2}). 
We have a similar characterization, that is, if $A$ is $\sigma$-unital, then 
the fundamental group of $A$ is equal to the set of trace scaling constants for automorphisms 
of $A\otimes\mathbb{K}$. 

We denote by $\mathcal{Z}$ the Jiang-Su algebra constructed in \cite{JS}. 
The Jiang-Su algebra $\mathcal{Z}$ is a unital separable simple infinite-dimensional nuclear 
C$^*$-algebra whose K-theoretic invariant is isomorphic to that of complex numbers. 
We may regard $\mathcal{Z}$ as the stably finite analogue of the Cuntz algebra $\mathcal{O}_{\infty}$. 
We say that a C$^*$-algebra $A$ is $\mathcal{Z}$-\textit{stable} if $A$ is isomorphic to 
$A\otimes\mathcal{Z}$. 
It has recently become important to study regularity properties in Elliott's classification program 
for nuclear C$^*$-algebras. 
In particular, Toms and Winter conjectured that for simple separable nuclear non-type I unital C$^*$-algebras, 
the properties of (i) finite nuclear dimension, (ii) $\mathcal{Z}$-stability and 
(iii) strict comparison of positive elements are equivalent (see, for example \cite{Toms} and \cite{Win}). 
It is known that (i) implies (ii) and (ii) implies (iii) due to work of Winter \cite{Win} and R\o rdam \cite{Ror} respectively. 
Recently, Matui and Sato showed that (iii) implies (ii) in the case of finitely many extremal tracial states in \cite{MS}. 

In this paper we shall compute Picard groups of several nuclear and non-nuclear 
simple stably projectionless C$^*$-algebras. 
In the case of stably projectionless C$^*$-algebras, 
the theory of the Cuntz semigroup enables us to compute Picard groups of 
several examples. 
We shall show that if $A$ is a separable simple exact $\mathcal{Z}$-stable 
stably projectionless C$^*$-algebra with a unique tracial state $\tau$ and no unbounded trace, then the following sequence is exact: 
\[\begin{CD}
      {1} @>>> \mathrm{Out}(A) @>>> \mathrm{Pic}(A) @>>> \mathcal{F}(A)
 @>>> {1} \end{CD}. \] 
Since there exists a unital simple $\mathcal{Z}$-stable algebra $A$ with a 
unique tracial state such that $\mathrm{Out}(A)$ is not a normal subgroup of 
$\mathrm{Pic}(A)$, $\mathcal{Z}$-stable stably projectionless C$^*$-algebras are in this sense more 
well-behaved than unital stably finite $\mathcal{Z}$-stable C$^*$-algebras. 
Let $\mathcal{W}_{2}$ be the Razak-Jacelon algebra studied in \cite{J}, \cite{Rob}, which has trivial K-groups and a unique tracial state and no unbounded trace. 
Then $\mathcal{W}_{2}$ is $\mathcal{Z}$-stable, and hence the sequence above is exact 
in this case. Moreover we shall show that the exact sequence above splits. 
Therefore $\mathrm{Pic}(\mathcal{W}_{2})$ is isomorphic to 
$\mathrm{Out}(\mathcal{W}_{2}) \rtimes \mathbb{R}_{+}^\times$. 

Based on the result of Matui and Sato, and Kirchberg's central sequence algebras, for any separable simple infinite-dimensional non-type I nuclear C$^*$-algebra with 
a finite dimensional lattice of densely defined lower semicontinuous traces, 
we shall show that $\mathcal{Z}$-stability and strict comparison are equivalent. 
(It is important to consider property (SI).) 

In particular, if $A$ is  a simple C$^*$-algebra with 
a finite dimensional lattice of densely defined lower semicontinuous traces in the class of 
Robert's classification theorem (\cite[Corollary 6.2.4]{Rob}), then $A$ is $\mathcal{Z}$-stable. 
Moreover we see that there are many examples that the sequence above is exact. 
But we do not know whether the exact sequence above splits in this case. 
This question is related to the existence of a one parameter trace scaling automorphism 
group of $A\otimes \mathbb{K}$. 
In the final part of this paper we shall give some remarks and a reason of the notation of $\mathcal{W}_{2}$. 
Some results show every separable simple $\mathcal{Z}$-stable 
stably projectionless  C$^*$-algebra $A$ with a unique tracial state has similar properties of (McDuff) II$_1$ factors. 

\section{The Picard group}\label{sec:Picard}
In this section we shall review basic facts on the Picard groups of 
C$^*$-algebras introduced by Brown, Green and Rieffel 
in \cite{BGR} and some results in \cite{Na}. 

Let $A$ be a C$^*$-algebra and $\mathcal{X}$ a right Hilbert $A$-module, 
and let $\mathcal{H}(A)$ denote the 
set of isomorphic classes $[\mathcal{X}]$ of 
countably generated right Hilbert $A$-modules. 
We denote by $L_A(\mathcal{X})$ 
the algebra of the adjointable operators on $\mathcal{X}$. 
For $\xi,\eta \in \mathcal{X}$, a  "rank one operator" $\Theta_{\xi,\eta}$ 
is defined by $\Theta_{\xi,\eta}(\zeta) 
= \xi \langle\eta,\zeta\rangle_A$ for $\zeta \in \mathcal{X}$. 
We denote by $K_A(\mathcal{X})$ the closure 
of the linear span of "rank one operators" $\Theta_{\xi,\eta}$ 
and by $\mathbb{K}$ the C$^*$-algebra of compact operators on an infinite-dimensional 
separable Hilbert space. 
Let $\mathcal{X}_A$ be a right Hilbert $A$-module $A$ with the obvious right $A$-action and 
$\langle a ,b\rangle_A = a^*b$ for $a,b \in A$. 
Then there exists a natural isomorphism of $K_A(\mathcal{X}_A)$ 
to $A$, where $A$ acts on $\mathcal{X}_A$ by left multiplication. 
Hence if $A$ is unital, then $K_A(\mathcal{X}_A)=L_A(\mathcal{X}_A)$. 
A multiplier algebra, denote by $M(A)$, of a C$^*$-algebra $A$ is 
the largest unital C$^*$-algebra that contains $A$ as an essential ideal. 
It is unique up to isomorphism over $A$ and isomorphic to $L_A(\mathcal{X}_A)$. 
Let $H_A$ denote the standard Hilbert module 
$\{(x_n)_{n\in \mathbb{N}}\; |\;x_n\in A,\sum x_n^*x_n\;\mathrm{converges}\; \mathrm{in}\; A\}$ 
with an $A$-valued inner product 
$\langle (x_n)_{n\in\mathbb{N}},(y_n)_{n\in\mathbb{N}}\rangle =\sum x_n^*y_n$. 
Then there exists a natural isomorphism of $A\otimes\mathbb{K}$ to 
$K_A(H_A)$. 

Let $A$ and $B$ be C$^*$-algebras. 
An $A$-$B$-{\it equivalence bimodule} is an $A$-$B$-bimodule $\mathcal{F}$ which is 
simultaneously a 
full left Hilbert $A$-module under a left $A$-valued inner product $_A\langle\cdot ,\cdot\rangle$ 
and a full right Hilbert $B$-module under a right $B$-valued inner product $\langle\cdot ,\cdot\rangle_B$, 
satisfying $_A\langle\xi ,\eta\rangle\zeta =\xi\langle\eta ,\zeta\rangle_B$ for any 
$\xi, \eta, \zeta \in \mathcal{F}$. We say that $A$ is {\it Morita equivalent} to $B$ 
if there exists an $A$-$B$-equivalence bimodule. 
There exists an isomorphism $\varphi$ of $A$ to $K_B(\mathcal{F})$ such that $\varphi ( _A\langle \xi ,\eta \rangle )=\Theta_{\xi ,\eta}$ 
for any $\xi ,\eta\in \mathcal{F}$. 
The standard Hilbert module $H_A$ can be regard as an $A\otimes\mathbb{K}$-$A$-equivalence bimodule. 
A dual module $\mathcal{F}^*$ of an $A$-$B$-equivalence bimodule $\mathcal{F}$ is a set 
$\{\xi^* ;\xi\in\mathcal{F} \}$ with the operations such that $\xi^* +\eta^*=(\xi +\eta )^*$, 
$\lambda\xi ^*=(\overline{\lambda}\xi)^*$, $b\xi^* a=(a^*\xi b^*)^*$, 
$_B\langle\xi^*,\eta^*\rangle =\langle\eta ,\xi\rangle_B$ and 
$\langle \xi^*,\eta^*\rangle_A =\;_A\langle\eta ,\xi\rangle$. 
The bimodule $\mathcal{F}^*$ is a $B$-$A$-equivalence bimodule. 
We refer the reader to \cite{RW} and \cite{R2} for the basic facts on 
equivalence bimodules and Morita equivalence. 
For $A$-$A$-equivalence bimodules 
$\mathcal{E}_1$ and 
$\mathcal{E}_2$, we say that $\mathcal{E}_1$ is isomorphic to $\mathcal{E}_2$ as an equivalence 
bimodule if there exists a $\mathbb{C}$-linear one-to-one map $\Phi$ of $\mathcal{E}_1$ onto 
$\mathcal{E}_2$ with the properties such that $\Phi (a\xi b)=a\Phi (\xi )b$, 
$_A\langle \Phi (\xi ) ,\Phi(\eta )\rangle =\;_A\langle \xi ,\eta\rangle$ and 
$\langle \Phi (\xi ) ,\Phi(\eta )\rangle_A =\langle\xi,\eta\rangle_A$ for $a,b\in A$, 
$\xi ,\eta\in\mathcal{E}_1$. 
The set of isomorphic classes $[\mathcal{E}]$ of the $A$-$A$-equivalence 
bimodules $\mathcal{E}$ forms a group under the product defined by 
$[\mathcal{E}_1][\mathcal{E}_2] = [\mathcal{E}_1 \otimes_A \mathcal{E}_2]$. 
We call it the {\it Picard group} of $A$ and denote it  by  $\mathrm{Pic}(A)$. 
The identity of $\mathrm{Pic}(A)$ is given by 
the $A$-$A$-bimodule $\mathcal{E}:= A$ with  
$\; _A\langle a_1 ,a_2 \rangle = a_1a_2^*$ and $\langle a_1 ,a_2\rangle_A = a_1^*a_2$ for 
$a_1,a_2 \in A$. The inverse element of $[\mathcal{E}]$ in the Picard group of $A$ 
is the dual module $[\mathcal{E}^*]$. 
Let $\alpha$ be an automorphism of $A$, and let 
$\mathcal{E}_{\alpha}^A=A$ with the obvious left $A$-action and the obvious $A$-valued inner product. 
We define the right $A$-action on $\mathcal{E}_\alpha^A$ by 
$\xi\cdot a=\xi\alpha(a)$ for 
any $\xi\in\mathcal{E}_\alpha^A$ and $a\in A$, and the right $A$-valued inner product by 
$\langle\xi ,\eta\rangle_A=\alpha^{-1} (\xi^*\eta)$ for any $\xi ,\eta\in\mathcal{E}_\alpha^A$.
Then $\mathcal{E}_{\alpha}^A$ is an $A$-$A$-equivalence bimodule. For $\alpha, \beta\in\mathrm{Aut}(A)$, 
$\mathcal{E}_\alpha^A$ is isomorphic to $\mathcal{E}_\beta^A$ if and only if 
there exists a unitary $u \in M(A)$ such that 
$\alpha = ad \ u \circ \beta $. Moreover, ${\mathcal E}_\alpha^A \otimes 
{\mathcal E}_\beta^A$ is 
isomorphic to $\mathcal{E}_{\alpha\circ\beta}^A$. Hence we obtain an homomorphism $\rho_A$ 
of $\mathrm{Out}(A)$ to $\mathrm{Pic}(A)$. 
Note that for any $\alpha\in\mathrm{Aut}(A)$, $\mathcal{E}_\alpha^A$ is isomorphic to $\mathcal{X}_A$ as a right Hilbert $A$-module. 
Conversely we have the following proposition. 

\begin{pro}\label{pro:trivial right module}
Let $\mathcal{E}$ be an $A$-$A$-equivalence bimodule such that 
$\mathcal{E}$ is isomorphic to $\mathcal{X}_A$ as a right Hilbert $A$-module. 
Then there exists an automorphism $\alpha$ of $A$ such that 
$\mathcal{E}$ is isomorphic to ${\mathcal E}_\alpha^A$ as an $A$-$A$-equivalence bimodule. 
\end{pro}
\begin{proof}
Let $\Phi$ be a right Hilbert $A$-module isomorphism of $\mathcal{X}_{A}$ to $\mathcal{E}$, and let 
$\psi$ be an isomorphism of $K_A(\mathcal{E})$ to $K_A(\mathcal{X}_{A})$ induced by $\Phi$. 
Since $K_A(\mathcal{X}_{A})$ is naturally isomorphic to $A$, we may regard $\psi$ as an isomorphism of 
$K_A(\mathcal{E})$ to $A$. 
There exists an isomorphism $\varphi$ of $A$ to $K_A(\mathcal{E})$ such that 
$\varphi (_A\langle \xi ,\eta \rangle) =\Theta_{\xi ,\eta}$ for any $\xi ,\eta\in \mathcal{E}$ because 
$\mathcal{E}$ is an $A$-$A$-equivalence bimodule. 

Put $\alpha :=(\psi \circ \varphi )^{-1}$, and define a map $F$ of $\mathcal{E}$ to $\mathcal{E}_{\alpha}^A$ by 
$F(\Phi (a) ):=\alpha (a)$ for any $a\in A$. 
Note that we have 
$$
_A \langle \Phi (a),\Phi (b)\rangle =\varphi ^{-1} (\Theta_{\Phi (a),\Phi (b)})= \varphi^{-1}\circ \psi^{-1} (ab^*) =\alpha (ab^*)
$$
and 
$$
a\cdot \Phi (b)=\varphi (a)\Phi (b) =\Phi (\psi \circ \varphi (a)b) =\Phi (\alpha^{-1}(a)b)
$$
for any $a,b\in A$. Therefore it can easily be checked that $F$ is an $A$-$A$-equivalence bimodule isomorphism. 
\end{proof}

An $A$-$B$-equivalence bimodule $\mathcal{F}$ induces an isomorphism $\Psi$ 
of $\mathrm{Pic}(A)$ to $\mathrm{Pic}(B)$ by 
$\Psi ([\mathcal{E}])=[\mathcal{F}^*\otimes\mathcal{E}\otimes\mathcal{F}]$ 
for $[\mathcal{E}]\in\mathrm{Pic}(A)$. 
Therefore if $A$ is Morita equivalent to $B$, then $\mathrm{Pic}(A)$ is isomorphic to 
$\mathrm{Pic}(B)$. Brown, Green and Rieffel showed that if $A$ is $\sigma$-unital, 
then $\mathrm{Pic}(A)$ is isomorphic to $\mathrm{Out}(A\otimes\mathbb{K})$ 
(see \cite[Theorem 3.4 and Corollary 3.5]{BGR}). 
Indeed a homomorphism $\rho_{A\otimes\mathbb{K}}$ of $\mathrm{Aut}(A\otimes\mathbb{K})$ to $\mathrm{Pic}(A\otimes\mathbb{K})$ induces 
an isomorphism of $\mathrm{Out}(A\otimes\mathbb{K})$ onto $\mathrm{Pic}(A\otimes\mathbb{K})$.

A sequence $\{\xi_{i}\}_{i\mathbb{N}}$ of a right Hilbert $A$-module $\mathcal{X}$ is called 
\textit{countable basis} of $\mathcal{X}$ if $\eta =\sum_{i=1}^\infty\xi_i\langle\xi_i,\eta\rangle_A$ in norm 
for any $\eta\in\mathcal{X}$. 
If $K_A(\mathcal{X})$ is $\sigma$-unital, then $\mathcal{X}$ has a countable basis. 
A sequence $\{\xi_{i}\}_{i\mathbb{N}}$ is a countable basis if and only if $\{\sum_{i=1}^N\Theta_{\xi_i,\xi_i}\}_{N\in\mathbb{N}}$ is 
an approximate unit for $K_A(\mathcal{X})$. 
See \cite{KPW}, \cite{KW}, \cite{Na} and \cite{W} for details of bases of Hilbert modules. 
We denote by $T(A)$ the set of densely defined lower semicontinuous traces on $A$ and 
$T_1(A)$ the set of tracial states on $A$. 
We have the following proposition. 
\begin{pro}(\cite[Proposition 2.4]{Na}) \\
Let $A$ be a simple $\sigma$-unital C$^*$-algebra and $\mathcal{X}$ a countably generated Hilbert $A$-module, and let $\tau$ be a densely defined lower 
semicontinuous trace on $A$. 
For $x\in K_A(\mathcal{X})_+$, 
define 
\[Tr_\tau^\mathcal{X}(x):=\sum_{i=1}^\infty \tau (\langle \xi_i,x\xi_i\rangle_A) \]
where $\{\xi_i\}_{i=1}^\infty$ is a countable basis of $\mathcal{X}$. 
Then $Tr_\tau^\mathcal{X}$ does not depend on the choice of basis and is a densely defined (resp. strictly 
densely defined) lower semicontinuous trace on $K_A(\mathcal{X})$ (resp. $L_A(\mathcal{X})$). 
\end{pro}
The following proposition is \cite[Remark 2.5]{Na}. Moreover it is well-known (see for example \cite{CZ}). But we include the proof for completeness. 
\begin{pro}\label{pro:trace correspondence}
Let $A$ be a simple $\sigma$-unital C$^*$-algebra and $\mathcal{X}$ a countably generated Hilbert $A$-module. 
Then there exists a bijective correspondence between $T(A)$ and $T(K_A(\mathcal{X}))$. 
\end{pro}
\begin{proof}
Since a right Hilbert $A$-module $\mathcal{X}$ is a $K_A(\mathcal{X})$-$A$-equivalence bimodule, 
$\mathcal{X}^*$ is an $A$-$K_A(\mathcal{X})$-equivalence bimodule. 
Let $\{\xi_j\}_{j\in\mathbb{N}}$ be a countable basis of $\mathcal{X}$ and 
$\{\eta^*_i\}_{i\in\mathbb{N}}$ a countable basis of $\mathcal{X}^*$. 
For any $a\in A_{+}$ and $\tau\in T(A)$, we have 
\begin{align*}
Tr_{Tr_\tau^\mathcal{X}}^{\mathcal{X}^*}(a)
& =\lim_{n\rightarrow\infty}\sum_{i=1}^n Tr_{\tau}^\mathcal{X} (\langle\eta^*_i,a\eta^*_i
\rangle_{K_A(\mathcal{X})}) \\
& =\lim_{n\rightarrow\infty}\sum_{i=1}^n Tr_\tau^\mathcal{X} (_{K_A(\mathcal{X})}
\langle \eta_i a^{\frac{1}{2}},\eta_i a^{\frac{1}{2}}\rangle ) \\
& =\lim_{n\rightarrow\infty}\sum_{i=1}^n Tr_\tau^\mathcal{X} (\Theta_{\eta_i a^{\frac{1}{2}},
\eta_i a^{\frac{1}{2}}}) \\
& =\lim_{n\rightarrow\infty}\sum_{i=1}^n\sum_{j=1}^\infty\tau (
\langle \xi_j, \eta_i a^{\frac{1}{2}}\langle \eta_i a^{\frac{1}{2}},\xi_j \rangle_A \rangle_A ) \\
& =\lim_{n\rightarrow\infty}\sum_{i=1}^n\sum_{j=1}^\infty\tau (
\langle\xi_j ,\eta_i a^{\frac{1}{2}} \rangle_A \langle \eta_i a^{\frac{1}{2}}, \xi_j\rangle_A ) \\
& =\lim_{n\rightarrow\infty}\sum_{i=1}^n\sum_{j=1}^\infty\tau (
\langle \eta_i a^{\frac{1}{2}}, \xi_j\rangle_A \langle\xi_j ,\eta_i a^{\frac{1}{2}} \rangle_A) \\
& =\lim_{n\rightarrow\infty}\sum_{i=1}^n \sum_{j=1}^\infty \tau (
\langle\eta_i a^{\frac{1}{2}},\xi_j\langle\xi_j ,\eta_i a^{\frac{1}{2}}\rangle_A \rangle_A ). 
\end{align*}
Since $\{\xi_j\}_{j\in\mathbb{N}}$ is a countable basis of $\mathcal{X}$, we see that 
$$\langle\eta_i a^{\frac{1}{2}},\sum_{j=1}^m\xi_j\langle\xi_j ,\eta_i a^{\frac{1}{2}}\rangle_A 
\rangle_A \nearrow \langle\eta_i a^{\frac{1}{2}},\eta_i a^{\frac{1}{2}}\rangle_A\;\;
(m\rightarrow \infty).
$$ 
Hence 
$$\lim_{n\rightarrow\infty}\sum_{i=1}^n \sum_{j=1}^\infty \tau (
\langle\eta_i a^{\frac{1}{2}},\xi_j\langle\xi_j ,\eta_i a^{\frac{1}{2}}\rangle_A \rangle_A )
=\lim_{n\rightarrow\infty}\sum_{i=1}^n \tau 
(\langle\eta_i a^{\frac{1}{2}},\eta_i a^{\frac{1}{2}}\rangle_A)$$
by the lower semicontinuity of $\tau$. 
Since $\Theta_{\eta_i^*,\eta_i^*}$ is corresponding to $\langle \eta_i,\eta_i\rangle_A$, 
we see that $\{\sum_{i=1}^n\langle \eta_i,\eta_i\rangle_A \}_{n\in\mathbb{N}}$ is an approximate unit 
for $A$. 
Therefore we have 
$$
\lim_{n\rightarrow\infty}\sum_{i=1}^n \tau 
(\langle\eta_i a^{\frac{1}{2}},\eta_i a^{\frac{1}{2}}\rangle_A)
=\lim_{n\rightarrow\infty}\tau (a^{\frac{1}{2}}(\sum_{i=1}^n\langle \eta_i,\eta_i\rangle_A
)a^{\frac{1}{2}})=\tau (a)
$$
by the lower semicontinuity of $\tau$. 
Consequently $Tr_{Tr_\tau^\mathcal{X}}^{\mathcal{X}^*}=\tau$. 

Since $(\mathcal{X}^*)^*$ is naturally isomorphic to $\mathcal{X}$ as a $K_A(\mathcal{X})$-$A$-equivalence bimodule, 
we see that $Tr_{Tr_\tau^{\mathcal{X}^*}}^{\mathcal{X}}=\tau$ for any $\tau\in T(K_A(\mathcal{X}))$ as above. 
Hence we obtain the conclusion. 
\end{proof}
The following Corollary is folklore. 
\begin{cor}\label{cor:folklore trace}
Let $A$ be a simple $\sigma$-unital C$^*$-algebra and $h$ be a non-zero positive element in $A$. 
Then every densely defined lower semicontinuous trace on $\overline{hAh}$ is a restriction of some densely defined lower semicontinuous trace on $A$. 
\end{cor}
\begin{proof}
Put $\mathcal{X}=\overline{hA}$. 
Since $K_A(\overline{hA})$ is naturally isomorphic to $\overline{hAh}$, it is enough to show that 
$Tr_\tau^\mathcal{X} (\tau) =\tau |_{\overline{hAh}}$ for any $\tau\in T(A)$ by Proposition \ref{pro:trace correspondence}. 
Let $\{a_n\}_{n\in \mathbb{N}}$ be a countable basis of $\overline{hA}$. 
(Note that the norm of Hilbert $A$-module $\overline{hA}$ is equal to the norm of C$^*$-algebra $A$.)
Since $\tau (a_n^*xa_n)=\tau (x^{\frac{1}{2}}a_na_n^*x^{\frac{1}{2}})$ for any $x\in \overline{hAh}_{+}$ and $\tau\in T(A)$, 
we have $\sum_{n=1}^{N} \tau (a_n^*xa_n)\leq \sum_{n=1}^{N+1} \tau (a_n^*xa_n)$ for any $N\in\mathbb{N}$. 
Therefore we see that $Tr_{\tau}^\mathcal{X} (x) =\tau (x)$ by the lower semicontinuity of $\tau$. 
\end{proof}
For $\tau\in T(A)$, define a map $\hat{T}_{\tau}$ of $\mathcal{H}(A)$ to $[0,\infty]$ by 
$$
\hat{T}_{\tau} ([\mathcal{X}]) :=\sum_{n=1}^\infty \tau (\langle \xi_n,\xi_n\rangle_A) 
$$
where $\{\xi_n\}_{n\in\mathbb{N}}$ is a countable basis of $\mathcal{X}$. 
This map is well-defined map and does not depend on the choice of basis. Moreover 
we have $\hat{T}_{\tau}(\mathcal{X})= Tr_{\tau}^\mathcal{X}(1_{L_A(\mathcal{X})})=\| Tr_{\tau}^\mathcal{X}\|$. 

Put $d_{\tau} (h)=\lim_{n\rightarrow \infty}\tau\otimes\mathrm{Tr} (h^{\frac{1}{n}})$ for 
$h\in (A\otimes \mathbb{K})_{+}$. Then $d_{\tau}$ is a dimension function. 
We have the following proposition. 
\begin{pro}\label{propositon3.1:Na}
\cite[Proposition 3.1]{Na} \\
Let $A$ be a simple $\sigma$-unital $C^*$-algebra with unique (up to scalar multiple) 
densely defined lower semicontinuous trace $\tau$ and 
$h$ a positive element in $A\otimes \mathbb{K}$. 
Then $\hat{T}_{\tau} (\overline{hH_A})=d_{\tau}(h)$. 
\end{pro}
The following proposition is an immediate corollary of \cite[Proposition 3.3]{Na}
\begin{pro}\label{proposition3.3:Na}
Let $A$ be a simple $\sigma$-unital C$^*$-algebra with a unique tracial state $\tau$ and no unbounded trace. 
Then for every right Hilbert $A$-module $\mathcal{X}$ and every $A$-$A$-equivalence bimodule $\mathcal{E}$,
\[\hat{T}_{\tau}([\mathcal{X}\otimes\mathcal{E}])=\hat{T}_{\tau}([\mathcal{X}])
\hat{T}_{\tau}([\mathcal{E}]).\]
\end{pro}
If $A$ is $\sigma$-unital, then for any $A$-$A$-equivalence bimodule 
$\mathcal{E}$ there exists a positive element $h$ in $A\otimes\mathbb{K}$ 
such that $\mathcal{E}$ is isomorphic to $\overline{hH_A}$ as a right Hilbert 
$A$-module. Note that $\overline{h(A\otimes\mathbb{K})h}$ is 
isomorphic to $A$ and $\overline{hH_A}$ has a suitable structure as an $A$-$A$-equivalence 
bimodule in this case. (See, for example, \cite[Proposition 2.3]{Na}.) 
The following proposition is a key proposition in this paper. 
\begin{pro}\label{pro:key fundamental}
Let $A$ be a simple $\sigma$-unital C$^*$-algebra with a unique tracial state $\tau$ and no unbounded trace. 
Define a map $T$ of $\mathrm{Pic}(A)$ to $\mathbb{R}_{+}^\times$ by 
$T([\overline{hH_A}])=d_{\tau} (h)$. 
Then $T$ is a well-defined multiplicative map and $T([\mathcal{E}_{\alpha}^A])=1$ for any $\alpha\in\mathrm{Aut}(A)$. Moreover 
$\mathrm{Im}(T)$ is equal to the set 
$$
\{d_\tau (h) \in \mathbb{R}^{\times}_{+}\ | \ 
h \text{ is a positive element in }  A\otimes\mathbb{K} \text{ such that } \\
A \cong \overline{h(A\otimes\mathbb{K})h}
\}.
$$
\end{pro}
\begin{proof}
Let $[\overline{hH_A}]\in \mathrm{Pic}(A)$. 
Then $d_{\tau}(h) =\hat{T}_{\tau}(\overline{hH_A})= \| Tr_{\tau}^{\overline{hH_A}}\| $ by Proposition \ref{propositon3.1:Na}. 
Since $K_A(\overline{hH_A})\cong A$ has no unbounded trace, $d_{\tau}(h)=\| Tr_{\tau}^{\overline{hH_A}}\| <\infty$. 
Hence we see that $T$ is well-defined map and $\mathrm{Im}(T)$ is equal to the set 
$$
\{d_\tau (h) \in \mathbb{R}^{\times}_{+}\ | \ 
h \text{ is a positive element in }  A\otimes\mathbb{K} \text{ such that } \\
A \cong \overline{h(A\otimes\mathbb{K})h}
\}
$$
by an argument above. 
Proposition \ref{proposition3.3:Na} implies that $T$ is a multiplicative map. 
It is easy to see that $\mathcal{E}_{\alpha}^A$ is isomorphic to $\overline{sA}=A$ as a right Hilbert $A$-module where $s$ is a 
strictly positive element in $A$. 
Since $\tau$ is a tracial state, $T([\mathcal{E}_{\alpha}^A])=d_{\tau} (s)=1$. 
\end{proof}
Put $\mathcal{F}(A)=\mathrm{Im}(T)$. 
We call $\mathcal{F}(A)$ the \textit{fundamental group} of $A$, which is a 
multiplicative subgroup of $\mathbb{R}^{\times}_{+}$ by the proposition above. 

Let $A$ a simple C$^*$-algebra with unique (up to scalar multiple) 
densely defined lower semicontinuous trace $\tau$. 
For any $\alpha\in\mathrm{Aut}(A\otimes\mathbb{K})$, $\tau\otimes\mathrm{Tr} \circ \alpha$ is a densely defined lower semicontinuous 
trace on $A\otimes\mathbb{K}$. 
Hence there exists a positive number $\lambda$ such that $\tau\otimes\mathrm{Tr} \circ \alpha =\lambda\tau\otimes\mathrm{Tr}$. 
Define a map $S$ of $\mathrm{Out}(A\otimes\mathbb{K})$ to $\mathbb{R}_{+}^\times$ by $S([\alpha ])=\lambda$ where 
$\tau\otimes\mathrm{Tr} \circ \alpha =\lambda\tau\otimes\mathrm{Tr}$. 
Then $S$ is a well-defined multiplicative map and $\mathrm{Im}(S)$ is equal to the set 
$$
\mathfrak{S}(A)
:= \{ \lambda \in \mathbb{R}^{\times}_+ \ | \ 
\tau\otimes\mathrm{Tr} \circ \alpha = \lambda \tau\otimes\mathrm{Tr} \text{ for some  } 
 \alpha \in \mathrm{Aut} (A\otimes \mathbb{K}) \ \}. 
$$
The following proposition is a strengthened version of \cite[Proposition 4.20]{Na}. 
\begin{pro}\label{pro:key spliting}
Let $A$ be a simple $\sigma$-unital C$^*$-algebra with a unique tracial state $\tau$ and no unbounded trace. 
Then there exists an isomorphism $\Psi$ of 
$\mathrm{Out}(A\otimes\mathbb{K})$ to $\mathrm{Pic}(A)$ such that $S([\alpha])^{-1} =T\circ \Psi ([\alpha])$ for any 
$[\alpha ]\in\mathrm{Out}(A\otimes\mathbb{K})$, and hence $\mathcal{F}(A)=\mathfrak{S}(A)$. 
\end{pro}
\begin{proof}
Define $\Psi ([\alpha]) :=[(H_A)^*\otimes\mathcal{E}_{\alpha}^{A\otimes\mathbb{K}}\otimes H_A]$. 
Since $H_A$ is an $A\otimes\mathbb{K}$-$A$-equivalence bimodule and $\rho_{A\otimes\mathbb{K}}$ induces an isomorphism of 
$\mathrm{Out}(A\otimes\mathbb{K})$ to $\mathrm{Pic}(A\otimes\mathbb{K})$ by \cite[Corollary 3.5]{BGR}, 
$\Psi$ is an isomorphism of $\mathrm{Out}(A\otimes\mathbb{K})$ to $\mathrm{Pic}(A)$. 
Note that $(H_A)^*$ is naturally isomorphic to $\overline{(s\otimes e_{11})(A\otimes\mathbb{K})}$ as an $A$-$A\otimes\mathbb{K}$-equivalence bimodule 
where $s$ is a strictly positive element in $A$ and $e_{11}$ is a rank one projection in $\mathbb{K}$. 
It is easy to see that for any element $\zeta$ in an algebraic tensor product 
$(s\otimes e_{11})(A\otimes\mathbb{K})\odot \mathcal{E}_{\alpha}^{A\otimes\mathbb{K}} \odot H_A$, 
there exists an element $\xi$ in $H_A$ such that 
$$
\zeta =(s^{\frac{1}{2}}\otimes e_{11})\otimes (s^{\frac{1}{4}}\otimes e_{11})\otimes \alpha^{-1} (s^{\frac{1}{4}}\otimes e_{11})\xi.
$$
Therefore it can easily be checked that $(H_A)^*\otimes \mathcal{E}_{\alpha}^{A\otimes\mathbb{K}}\otimes H_A$ is isomorphic to 
$\overline{\alpha^{-1}(s\otimes e_{11})H_A}$ as a right Hilbert $A$-module. 
We have 
$$
d_{\tau} (\alpha^{-1}(s\otimes e_{11})) =
\lim_{n \to \infty}\tau\otimes\mathrm{Tr} (\alpha^{-1} ((s\otimes e_{11})^{\frac{1}{n}})) =S([\alpha ])^{-1}
$$
since $\tau$ is a tracial state on $A$. 
Hence we obtain the conclusion. 
\end{proof}

\section{The Cuntz semigroup}
In this section we shall review basic facts of the Cuntz semigroup and some results in \cite{CEI}, 
\cite{ERS}, \cite{Rob2} and \cite{Ror}. See, for example, \cite{APT} for details of the Cuntz semigroup. 
Let $A$ be a C$^*$-algebra. For positive elements $a,b\in A$ 
we say that $a$ is \textit{Cuntz smaller than} $b$, written $a\precsim b$, if there exists a 
sequence $\{x_n\}_{n\in\mathbb{N}}$ of $A$ such that $\| x_n^*bx_n-a\|\rightarrow 0$. 
Positive elements $a$ and $b$ are said to be \textit{Cuntz equivalent}, written $a \sim b$, if 
$a\precsim b$ and $b\precsim a$. Define the \textit{Cuntz semigroup} $\mathrm{Cu}(A)$ as the set of 
Cuntz equivalence classes of positive elements in $A\otimes\mathbb{K}$ endowed with the order 
$[a]\leq [b]$ if $a$ is Cuntz smaller than $b$, and the addition $[a]+[b]=
[a^{\prime}+b^{\prime}]$ where $a\sim a^{\prime}$, $b\sim b^{\prime}$ and 
$a^{\prime}b^{\prime}=0$. 
Note that this definition is different from the original definition $\mathrm{W}(A)$ in 
\cite{Cu}. (We have $\mathrm{Cu}(A)=\mathrm{W}(A\otimes\mathbb{K})$.) 
The Cuntz semigroup $\mathrm{Cu}(A)$ is also defined using right Hilbert 
$A$-modules (see \cite{CEI}). 
For positive elements $a,b\in A\otimes\mathbb{K}$ we say that 
$a$ is \textit{compactly contained} in $b$, written $a\ll b$ if whenever 
$[b]\leq \sup_{n\in\mathbb{N}} [b_n]$ for an increasing sequence $\{[b_n]\}_{n\in\mathbb{N}}$, then 
there exists a natural number $n$ such that $[a]\leq [b_n]$. 
Coward, Elliott and Ivanescu \cite{CEI} showed that $\mathrm{Cu}(A)$ has the following 
properties:\ \\
\quad (1) every increasing sequence in $\mathrm{Cu}(A)$ has a supremum, \ \\
\quad (2) for any element $[a]$ in $\mathrm{Cu}(A)$ there exists an increasing sequence 
$\{[a_n]\}_{n\in\mathbb{N}}$ of $\mathrm{Cu}(A)$ such that $[a_{n}]\ll [a_{n+1}]$ for 
any $n\in\mathbb{N}$ and $[a]=\sup [a_n]$, \ \\
\quad (3) the operation of passing to the supremum of an increasing sequence and the relation 
$\ll$ are compatible with addition. 

Moreover they showed that $\mathrm{Cu}(A)$ is a functor which is continuous 
with respect to inductive limits (\cite[Theorem 2]{CEI}). 
For a positive element $a\in A\otimes\mathbb{K}$ and $\epsilon >0$ 
we denote by $(a-\epsilon)_{+}$ the element $f(a)$ in $A\otimes\mathbb{K}$ 
where $f(t)=\max\{0,t-\epsilon \}, t\in \sigma (a)$. 
Then we have $(a-\epsilon )_{+}\ll a$. 

Following the definition in \cite{Ror}, the Cuntz semigroup $\mathrm{Cu}(A)$ is said to be 
\textit{almost unperforated} if $(k+1)[a] \leq k[b]$ for some $k\in\mathbb{N}$ implies that  
$[a]\leq [b]$. 
R\o rdam showed that if $A$ is $\mathcal{Z}$-stable, then 
$\mathrm{Cu}(A)$ is almost unperforated (see \cite[Theorem 4.5]{Ror}). 
If $A$ is a simple exact C$^*$-algebra with traces, then $\mathrm{Cu}(A)$ is almost unperforated 
if and only if $A$ has strict comparison, that is, 
if $a,b\in (A\otimes \mathbb{K})_{+}$ with $d_{\tau}(a)< d_{\tau} (b)< \infty$ for any 
$\tau \in T(A)$, then $[a]\leq [b]$. 
(See \cite[Proposition 4.2, Remark 4.3 and Proposition 6.2]{ERS} and \cite[Proposition 3.2 and Corollary 4.6]{Ror}.) 

\begin{lem}\label{lem:Brownsup}
Let $A$ be a simple C$^*$-algebra and $a$ a non-zero positive element in $A\otimes\mathbb{K}$. 
Then for any positive element $b$ in $A\otimes\mathbb{K}$, $[b]\leq \sup_{n\in\mathbb{N}}n[a]$. 
\end{lem}
\begin{proof}
Let $B:=\overline{(a+b)(A\otimes\mathbb{K})(a+b)}$. Then $B$ is a $\sigma$-unital hereditary subalgebra of $A\otimes\mathbb{K}$. 
By a variant of Brown's theorem (see for example \cite[Theorem 1.9]{MP}), $\overline{aH_B}^{\infty}$ is isomorphic to $H_B$ as a right Hilbert module. 
Since $\overline{bH_B}\subseteq H_B$, we see that $[b]\leq \sup_{n\in\mathbb{N}}n[a]$ in $\mathrm{Cu}(B)$. 
We obtain the conclusion because $B$ is a hereditary subalgebra of $A\otimes\mathbb{K}$. 
\end{proof}

The following proposition is an immediate corollary of \cite[Theorem 6.6]{ERS}. 
(Note that they considered the more general case.) 
But we shall give a self-contained proof based on their arguments 
(see also \cite[Proposition 6.4]{ERS}). 
\begin{pro}\label{cor:ERS}
Let $A$ be a simple exact C$^*$-algebra, and let $a$ and 
$b$ be positive elements in $A\otimes\mathbb{K}$. Assume that $\mathrm{Cu}(A)$ is 
almost unperforated and $0$ is an accumulation point of the spectrum $\sigma (a)$ of $a$. 
Then if $d_{\tau}(a)\leq d_{\tau}(b) < \infty$ for any $\tau\in T(A)$, 
then $a$ is Cuntz smaller than $b$. 
\end{pro}
\begin{proof}
Let $a$ and $b$ be positive elements in $A\otimes\mathbb{K}$ such that 
$d_{\tau}(a)\leq d_{\tau}(b)$ for any $\tau\in T(A)$. We may assume that 
$\| a\| =\| b\| =1$. 
For any $k\in\mathbb{N}$ we have 
$$
d_{\tau}(\mathrm{diag}(\overbrace{a,..,a}^{k}))=kd_{\tau}(a)\leq kd_{\tau} (b)
< (k+1)d_{\tau}(b)=d_{\tau}(\mathrm{diag}(\overbrace{b,..,b}^{k+1})).
$$ Hence $k[a]\leq (k+1)[b]$ for any $k\in\mathbb{N}$ because $A$ has strict comparison. 
Let $\epsilon >0$, and choose a positive function $c_{\epsilon}$ on $\sigma (a)$ 
such that $c_{\epsilon}(t)>0$ on $t\in (0,\epsilon)$ and $c_{\epsilon}(t)=0$ on 
$\sigma (a)\setminus (0,\epsilon )$. Then we have 
$[c_{\epsilon}(a)]+[(a-\epsilon)_{+}] \leq [a]$. 
Note that for any $\epsilon >0$, $c_{\epsilon}(a)$ is a nonzero positive element 
because $0$ is an accumulation point of $\sigma (a)$. 
Hence we have $2[a]\leq \sup_{n\in\mathbb{N}}n[c_{\epsilon}]$ by Lemma \ref{lem:Brownsup}. 
There exists a natural number $m$ such that $2[(a-\epsilon )_{+}]\leq m[c_{\epsilon}(a)]$ 
since $2[(a-\epsilon)_{+}]\ll 2[a]$. 
Therefore we have 
$$
(m+2)[(a-\epsilon)_{+}]\leq m[(a-\epsilon)_{+}]+m[c_{\epsilon}(a)]\leq m[a]\leq (m+1)[b]. 
$$
By the assumption that $\mathrm{Cu}(A)$ is almost unperforated, we see that 
$[(a-\epsilon)_{+}] \leq [b]$ for any $\epsilon >0$, and hence we have $[a]\leq [b]$. 
\end{proof}
\begin{cor}\label{pro:ERS}
Let $A$ be a simple exact stably projectionless C$^*$-algebra, and let $a$ and 
$b$ be positive elements in $A\otimes\mathbb{K}$. Assume that $\mathrm{Cu}(A)$ is 
almost unperforated. Then if $d_{\tau}(a)=d_{\tau}(b) <\infty$ for any $\tau\in T(A)$, 
then $a$ is Cuntz equivalent to $b$. 
\end{cor}
\begin{proof} 
For any nonzero positive element $a$ in $A\otimes\mathbb{K}$, $0$ is an accumulation point of $\sigma (a)$ 
because $A$ is a stably projectionless C$^*$-algebra. 
Hence we obtain the conclusion by Proposition \ref{cor:ERS}. 
\end{proof}
Based on the result in \cite{Rob2}, we say that a C$^*$-algebra $A$ 
\textit{has almost stable rank one} if for every $\sigma$-unital hereditary subalgebra 
$B\subseteq A\otimes\mathbb{K}$ we have $B\subseteq \overline{\mathrm{GL}(\widetilde{B})}$. 
Robert showed that if $A$ is a simple $\mathcal{Z}$-stable stably projectionless 
C$^*$-algebra, then $A$ has almost stable rank one 
(see \cite[Corollary 4.5]{Rob2} and \cite{Ror}). 
The following proposition is \cite[Proposition 4.7]{Rob2}. 
See \cite[Theorem 3]{CEI} for the proof. 
\begin{pro}\label{pro:CEIR}
Let $A$ be a simple $\sigma$-unital C$^*$-algebra such that $A$ has almost stable rank one and 
$a$ and $b$ positive elements in $A\otimes\mathbb{K}$. 
Then $a$ is Cuntz smaller than $b$ if and only if there exists a right Hilbert $A$-module 
$\mathcal{X}\subseteq \overline{bH_A}$ such that $\mathcal{X}$ is isomorphic 
to $\overline{aH_A}$ as a right Hilbert $A$-module, and 
$a$ is Cuntz equivalent to $b$ if and only if $\overline{aH_A}$ is isomorphic to 
$\overline{bH_A}$ as a right Hilbert $A$-module. 
\end{pro}
Corollary \ref{pro:ERS} and Proposition \ref{pro:CEIR} are important in the proof of 
our main result. These propositions show that every 
separable simple $\mathcal{Z}$-stable stably projectionless  C$^*$-algebra $A$ with a unique 
tracial state has similar properties of II$_1$ factors (Murray-von Neumann comparison theory). 
Moreover we have the following proposition. 
\begin{pro}\label{pro:hereditary}
Let $A$ be a simple exact $\sigma$-unital stably projectionless C$^*$-algebra with a unique tracial state $\tau$ and no unbounded trace. 
Assume that $\mathrm{Cu}(A)$ is almost unperforated, $A$ has almost stable rank one and $\mathcal{F}(A)=\mathbb{R}_{+}^\times$. 
Then every nonzero hereditary subalgebra of $A$ is isomorphic to $A$. 
\end{pro}
\begin{proof}
Let $B$ be a non-zero hereditary subalgebra of $A$.  Then $B$ is isomorphic to $\overline{h_0Ah_0}$ for some non-zero positive element $h_0$ in $A$. 
Since $d_{\tau}(h_0)\in \mathbb{R}_{+}^\times =\mathcal{F}(A)$, there exists a positive element $h$ in $A\otimes\mathbb{K}$ such that $d_{\tau}(h)=d_{\tau}(h_0)$ and 
$\overline{h(A\otimes\mathbb{K})h}$ is isomorphic to $A$. But then $h\sim h_0$ by Corollary \ref{pro:ERS} and so 
$\overline{hH_A}$ is isomorphic to $\overline{(h_0\otimes e_{11})H_A}$ by Proposition \ref{pro:CEIR}. 
Hence 
$
A\cong K_A(\overline{hH_A}) \cong K_A(\overline{(h_0\otimes e_{11})H_A}) \cong B
$.
\end{proof}

\section{Main result}
The following theorem is the main result in this paper. 
See \cite[Corollary 4.8]{kod1} and \cite[Proposition 3.26]{NW} for the unital case. 
\begin{thm}\label{thm:main result}
Let $A$ be a simple exact $\sigma$-unital stably projectionless C$^*$-algebra with a unique traical state $\tau$ and no unbounded trace. 
Assume that $\mathrm{Cu}(A)$ is almost unperforated and 
$A$ has almost stable rank one. Then 
the 
following sequence is exact:
\[\begin{CD}
      {1} @>>> \mathrm{Out}(A) @>\rho_A>> \mathrm{Pic}(A) @>T>> \mathcal{F}(A)
 @>>> {1} \end{CD}. \] 
\end{thm}
\begin{proof}
It is clear that $T$ is onto by definition of $\mathcal{F}(A)$. 
We see that $\rho_A$ is one-to-one and $\mathrm{Im}(\rho_A)\subseteq\mathrm{Ker}(T)$ 
by \cite[Corollary 3.2]{BGR} and Proposition \ref{pro:key fundamental} respectively. 
We shall show that $\mathrm{Ker}(T)\subseteq \mathrm{Im}(\rho_A)$. 
Let $[\mathcal{E}]\in \mathrm{Ker}(T)$. Then Corollary \ref{pro:ERS} and 
Proposition \ref{pro:CEIR} imply $\mathcal{E}$ is isomorphic to 
$\overline{(s\otimes e_{11})H_A}$ as a right Hilbert $A$-module where 
$s$ is a strict positive element in $A$ and $e_{11}$ is a rank one projection in 
$\mathbb{K}$ because we have $d_{\tau}(s\otimes e_{11})=1$ by $\| \tau \| =1$. 
Since $\overline{(s\otimes e_{11})H_A}$ is isomorphic to $\mathcal{X}_A$ as a right Hilbert $A$-module, 
there exists some automorphism $\alpha$ such that $[\mathcal{E}]=[\mathcal{E}_\alpha^A]$ by Proposition \ref{pro:trivial right module}. 
Hence $[\mathcal{E}]\in \mathrm{Im}(\rho_A)$. 
\end{proof}

\begin{cor}\label{cor:main}
Let $A$ be a simple exact separable $\mathcal{Z}$-stable stably projectionless 
C$^*$-algebra with a unique tracial state $\tau$ and no unbounded trace. 
Then 
the 
following sequence is exact: 
\[\begin{CD}
      {1} @>>> \mathrm{Out}(A) @>\rho_A>> \mathrm{Pic}(A) @>T>> \mathcal{F}(A)
 @>>> {1} \end{CD}. \] 
\end{cor}
\begin{proof}
This is an immediate consequence of \cite[Theorem 4.5]{Ror}, \cite[Corollary 4.5]{Rob2} and Theorem \ref{thm:main result}. 
\end{proof}
\begin{rem}\label{rem:non normal} 
There exists a unital simple AF algebra $A$ with a unique tracial state such that 
$\mathrm{Out}(A)$ is not a normal subgroup of $\mathrm{Pic}(A)$. (See \cite{Na2}.) 
Of course $A$ is a unital stably finite $\mathcal{Z}$-stable C$^*$-algebra. 
Therefore the corollary above shows that 
$\mathcal{Z}$-stable stably projectionless C$^*$-algebras are in this sense more well-behaved than 
unital stably finite $\mathcal{Z}$-stable C$^*$-algebras. 
\end{rem}
We shall show some examples. 

Let $\mathcal{W}_{2}$ be the Razak-Jacelon algebra studied in \cite{J}, \cite{Rob} and 
\cite{Rob2}, which has trivial K-groups and a unique tracial state and no unbounded trace. 
The Razak-Jacelon algebra $\mathcal{W}_{2}$ is constructed as an inductive limit C$^*$-algebra 
of Razak's building block in \cite{Raz}, that is,  
$$
A(n,m)= \left\{f\in C([0,1])\otimes M_m(\mathbb{C}) \ | \
\begin{array}{cc} 
f(0)=\mathrm{diag}(\overbrace{c,..,c}^k,0_{n}), 
f(1)=\mathrm{diag}(\overbrace{c,..,c}^{k+1}),  \\
c\in M_n(\mathbb{C})
\end{array} 
\right\}
$$
where $n$ and $m$ are natural numbers with $n|m$ and $k:=\frac{m}{n}-1$. 
Let $\mathcal{O}_{2}$ denote the Cuntz algebra generated by $2$ isometries 
$S_1$ and $S_2$. 
For every $\lambda_1,\lambda_2\in\mathbb{R}$ there exists by universality a one-parameter 
automorphism group $\alpha$ of $\mathcal{O}_2$ given by $\alpha_t (S_j)=e^{it\lambda_{j}}S_j$. 
Kishimoto and Kumjian showed that if 
$\lambda_{1}$ and $\lambda_{2}$ are all nonzero of the same sign and 
$\lambda_1$ and $\lambda_2$ generate $\mathbb{R}$ as a closed subgroup, then 
$\mathcal{O}_2\rtimes_{\alpha}\mathbb{R}$ is a simple stable projectionless C$^*$-algebra 
with unique (up to scalar multiple) densely defined lower semicontinuous trace 
in \cite{KK1} and \cite{KK2}. 
Moreover Robert \cite{Rob} showed that $\mathcal{W}_{2}\otimes \mathbb{K}$ is isomorphic to 
$\mathcal{O}_2\rtimes_{\alpha}\mathbb{R}$ for some $\lambda_1$ and $\lambda_2$. 
(See also \cite{Dean}.) In particular, $\mathcal{W}_{2}\otimes \mathbb{K}$ has a 
one parameter trace scaling automorphism group $\sigma$ (see \cite{KK1}). 

\begin{thm}\label{thm:example main}
The Picard group of Razak-Jacelon algebra $\mathcal{W}_{2}$ is isomorphic 
to a semidirect product  of $\mathrm{Out}(\mathcal{W}_{2})$ with $\mathbb{R}_{+}^\times$. 
Moreover if $A$ is a simple exact $\sigma$-unital C$^*$-algebra with 
a unique tracial state $\tau$ and no unbounded trace, then the Picard group of $A\otimes \mathcal{W}_2$ is isomorphic 
to a semidirect product  of $\mathrm{Out}(A\otimes \mathcal{W}_2)$ with $\mathbb{R}_{+}^\times$. 
\end{thm}
\begin{proof}
Note that we see that $A\otimes \mathcal{W}_2$ is stably projectionless C$^*$-algebra 
because $A\otimes\mathcal{W}_{2}\otimes\mathbb{K}$ has a 
one parameter trace scaling automorphism group $\mathrm{id}\otimes\sigma$. 
Since $\mathcal{W}_{2}$ is $\mathcal{Z}$-stable, we have the following exact sequence: 
\[\begin{CD}
      {1} @>>> \mathrm{Out}(A\otimes \mathcal{W}_2) @>\rho_A>> \mathrm{Pic}(A\otimes \mathcal{W}_2) @>T>> \mathcal{F}(A\otimes \mathcal{W}_2)
 @>>> {1} \end{CD} \] 
by Corollary \ref{cor:main}. 
By Proposition \ref{pro:key spliting}, we see that $\mathcal{F}(A\otimes\mathcal{W}_2)=\mathbb{R}_{+}^\times$ and the exact sequence above 
splits because $A\otimes\mathcal{W}_2\otimes\mathbb{K}$ has a one parameter trace scaling automorphism group. 
Consequently $\mathrm{Pic}(A\otimes\mathcal{W}_2)$ is isomorphic to $\mathrm{Out}(A\otimes\mathcal{W}_2)\rtimes\mathbb{R}_{+}^\times$. 
\end{proof} 

\begin{rem}
(i) Note that we have 
$$
\mathrm{Out}(\mathcal{W}_{2}\otimes\mathbb{K})\cong 
\mathrm{Out}(\mathcal{W}_{2})\rtimes \mathbb{R}_{+}^\times.
$$ 
(ii) We do not assume that $A$ is nuclear in the theorem above. 
Hence we have 
$$
\mathrm{Pic}(\mathcal{W}_{2}\otimes C_r^*(\mathbb{F}_n))\cong 
\mathrm{Out}(\mathcal{W}_{2}\otimes C_r^*(\mathbb{F}_n))\rtimes \mathbb{R}_{+}^\times 
$$ 
where $\mathbb{F}_n$ is a non-amenable free group with $n$ generators. Moreover 
Proposition \ref{pro:hereditary} shows that every nonzero hereditary subalgebra of 
$\mathcal{W}_{2}\otimes C_r^*(\mathbb{F}_n)$ is isomorphic to 
$\mathcal{W}_{2}\otimes C_r^*(\mathbb{F}_n)$. 
\ \\
(iii) Let $B$ be a simple unital AF algebra with two extremal tracial states. 
Then $\mathcal{W}_{2}\otimes B$ is a simple stably projectionless C$^*$-algebra with 
two extremal tracial states and in the class of Robert's classification theorem 
\cite{Rob}. It can be checked that $\mathrm{Out}(\mathcal{W}_{2}\otimes B)$ is 
not a normal subgroup of $\mathrm{Pic}(\mathcal{W}_{2}\otimes B)$ by 
Robert's classification theorem and a similar proposition as \cite[Proposition 1.5]{kod1}. 
(We need to replace the K$_0$-groups with the trace spaces.) 
\end{rem}

\section{$\mathcal{Z}$-stability of stably projectionless C$^*$-algebras}\label{sec:Z-stable}
In this section we shall generalize the result of Matui and Sato in \cite{MS} to stably projectionless C$^*$-algebras. 
Note that our arguments are essentially based on their arguments. 

We shall review some results of Kirchberg's central sequence algebra in \cite{Kir2}. We denote by  $\tilde{A}$ the unitization algebra of $A$. Note that we consider $A=\tilde{A}$ when $A$ is unital. 
For a separable C$^*$-algebra $A$, set 
$$
c_0(A):=\{(a_n)_{n\in\mathbb{N}}\in \ell^{\infty}(\mathbb{N}, A)\; |\; \lim_{n \to \infty}\| a_n\| =0 \}, \; 
A^{\infty}:=\ell^{\infty}(\mathbb{N}, A)/c_0(A). 
$$
Let $B$ be a C$^*$-subalgebra of $A$. 
We identify $A$ and $B$ with the C$^*$-subalgebras of $A^\infty$ consisting of equivalence classes of 
constant sequences.  Put 
$$
A_{\infty}:=A^{\infty}\cap A^{\prime},\; \mathrm{Ann}(B,A^{\infty}):=\{(a_n)_n\in A^{\infty}\cap B^{\prime}\; |\; (a_n)_nb =0
\;\mathrm{for}\;\mathrm{any}\; b\in B \}.
$$
Then $\mathrm{Ann}(B,A^{\infty})$ is an closed two-sided ideal of $A^{\infty}\cap B^{\prime}$, and define 
$$
F(A):=A_{\infty}/\mathrm{Ann}(A,A^{\infty}).
$$
We call $F(A)$ the \textit{central sequence algebra} of $A$. A sequence $(a_n)_n$ is said to be 
\textit{central} if $\lim_{n\to \infty}\| a_na-aa_n \| =0$ for all $a\in A$. A central sequence 
is a representative of an element in $A_{\infty}$. 
Since $A$ is separable, $A$ has a countable approximate unit $\{h_n\}_{n\in\mathbb{N}}$. 
It is easy to see that $[(h_n)_n]$ is a unit in $F(A)$. If $A$ is unital, then $F(A)=A_{\infty}$. 
Moreover we see that  $F(A)$ is isomorphic to $M(A)^\infty\cap A^{\prime}/\mathrm{Ann}(A,M(A)^{\infty})$ since 
for any $(y_n)_n\in M(A)^\infty\cap A^{\prime}$, $(y_nh_n)_n$ is a central sequence in $A$ and 
$[(y_n)_n]=[(y_nh_n)_n]$ in $M(A)^\infty\cap A^{\prime}/\mathrm{Ann}(A,M(A)^{\infty})$. 
Let $\{e_{ij}\}_{i,j\in\mathbb{N}}$ be the standard matrix units of $\mathbb{K}$. 
Define a map $\varphi$ of $F(A)$ to $F(A\otimes\mathbb{K})$ by $\varphi ([(x_n)_n])=[(x_n\otimes\sum_{i=1}^ne_{ii})_n]$. 
Then it is easily seen that $\varphi$ is a well-defined injective homomorphism. 
A similar argument as above shows any element in $F(A\otimes\mathbb{K})$ is equal to $[(\sum_{i,j=1}^nx_{n,i,j}\otimes e_{i,j})_n]$ for some 
sequence $\{x_{n,i,j}\}_{n\in\mathbb{N}}$ in $A$. 
Using matrix units and the centrality of sequence, we can show that if $i\neq j$, then $\lim_{n\to \infty}x_{n,i,j}a=0$ for any $a\in A$ and 
$\lim_{n\to\infty}(x_{n,i,i}-x_{n,j,j})a=0$ for any $i,j\in\mathbb{N}$ and $a\in A$. 
Since $M_{\infty} (A)$ is dense in $A\otimes\mathbb{K}$, it can be checked that $\varphi$ is surjective. 
Hence $F(A)$ is isomorphic to $F(A\otimes\mathbb{K})$. (See \cite[Proposition 1.9]{Kir2} for more general cases.) 

We denote by $I(k,k+1)$ the prime dimension drop algebra 
$$
\{f\in C([0,1])\otimes M_{k}(\mathbb{C})\otimes M_{k+1}(\mathbb{C})\; |\; f(0)\in M_{k}(\mathbb{C})
\otimes \mathrm{id}_{k+1}, f(1)\in \mathrm{id}_{k}\otimes M_{k+1}(\mathbb{C}) \}
$$
for $k\in\mathbb{N}$. 
The Jiang-Su algebra $\mathcal{Z}$ is constructed as an inductive limit 
C$^*$-algebra of prime dimension drop algebras in \cite{JS}. 
We shall show the following proposition (which is based on \cite [Proposition 2.2]{TW2}) by a similar way as in \cite[Theorem 7.2.2]{Ror1}. 
See \cite[Proposition 4.11]{Kir2} for more general cases. 
\begin{pro}\label{pro:z-stable kirchberg}
Let $A$ be a separable C$^*$-algebra. There exist a unital homomorphism of 
the prime dimension drop algebra $I(k,k+1)$ to $F(A)$ for any $k\in\mathbb{N}$ if and only if $A$ is $\mathcal{Z}$-stable. 
\end{pro}
\begin{proof}
Assume that there exist a unital homomorphism of 
the prime dimension drop algebra $I(k,k+1)$ to $F(A)$ for any $k\in\mathbb{N}$. 
By a similar argument as in \cite [Proposition 2.2]{TW2} and the construction of $\mathcal{Z}$ in \cite{JS}, 
we see that there exists a unital homomorphism $\alpha$ of $\mathcal{Z}$ to $F(A)$. 

Let $\varphi$ be an injective homomorphism of $A$ to $A\otimes\mathcal{Z}$ defined by $\varphi (a)=a\otimes 1_{\mathcal{Z}}$, and 
put $C:=M(A\otimes\mathcal{Z})^{\infty}\cap\varphi (A)^{\prime}/\mathrm{Ann}(\varphi (A),M(A\otimes\mathcal{Z})^{\infty})$. 
Then we can regard $\alpha$ as a unital homomorphism of $\mathcal{Z}$ to $C$ since $F(A)$ is isomorphic to 
$M(A)^\infty\cap A^{\prime}/\mathrm{Ann}(A,M(A)^{\infty})$. Define a unital homomorphism of $\beta$ of $\mathcal{Z}$ to 
$M(A\otimes\mathcal{Z})^{\infty}\cap\varphi (A)^{\prime}$ by $\beta (x)=(1_{M(A)}\otimes x)_n$, and let $[\beta ]: \mathcal{Z}\to C$ be the quotient homomorphism 
of $\beta$. 
Then we see that C$^*(\alpha (\mathcal{Z}),[\beta ](\mathcal{Z}))$ in $C$ is isomorphic to 
$\mathcal{Z}\otimes\mathcal{Z}$. Since $\mathcal{Z}$ has approximately inner flip and is K$_1$-injective (see \cite[Proposition 1.13]{TW1}), 
there exists a sequence $\{w_m\}_{m\in\mathbb{N}}$ of unitary elements in $C$ such that 
$\lim_{m\to\infty}w_m^*[\beta ](x)w_m=\alpha (x)$ for any $x\in \mathcal{Z}$ and $w_m$ is in the connected component of $1_{C}$ in $U(C)$ for any $m\in\mathbb{N}$. 
Since $w_m$ is in the connected component of $1_{C}$ in $U(C)$, there exists a unitary element $u_m$ in $M(A\otimes\mathcal{Z})^{\infty}\cap\varphi (A)^{\prime}$ 
such that $[u_m]=w_m$ for any $m\in\mathbb{N}$. 
For any $a\in A, x\in\mathcal{Z}$ and all $y\in M(A\otimes\mathcal{Z})^{\infty}\cap\varphi (A)^{\prime}$ such that $[y]=\alpha (x)$, we have 
$$
y\varphi (a)=\lim_{m\to\infty}u_m^*\beta (x)u_m\varphi (a)=\lim_{m\to\infty}u_m^*\beta (x)\varphi (a)u_m=\lim_{m\to\infty}u_m^*(a\otimes x)u_m
$$  
by $[y]=\lim_{m\to\infty}[u_m^*\beta (x)u_m]$ and the definition of $\mathrm{Ann}(\varphi (A),M(A\otimes\mathcal{Z})^{\infty})$. 
Since $[y]=\alpha (x)$, we can take $y\in M(\varphi (A))^\infty\cap \varphi (A)^{\prime}\subseteq M(A\otimes\mathcal{Z})^{\infty}\cap\varphi (A)^{\prime}$. 
Hence we see that $\lim_{m\to\infty}u_m^*(a\otimes x)u_m$ is an element in $\varphi (A)^{\infty}$. 
Therefore for any $z\in A\otimes \mathcal{Z}$, $\lim_{m\to \infty}d(u^*_mzu_m, \varphi (A)^{\infty})=0$. 
We see that $A$ is $\mathcal{Z}$-stable by a similar argument as in \cite[Proposition 2.3.5 and Proposition 7.2.1]{Ror1}. 

Conversely assume that $A$ is $\mathcal{Z}$-stable. 
Then $A$ is isomorphic to $A\otimes (\otimes_{k=1}^\infty \mathcal{Z})$. 
Since $M(A)$ is the largest unital C$^*$-algebra that contains $A$ as an essential ideal, 
$\tilde{A}\otimes (\otimes_{k=1}^\infty \mathcal{Z})$ is a unital subalgebra of  $M(A\otimes(\otimes_{k=1}^\infty \mathcal{Z}))$. 
Hence there exists a unital homomorphism of $\mathcal{Z}$ to $M(A)^\infty\cap A^{\prime}$. 
Therefore we see that there exists a unital homomorphism of 
the prime dimension drop algebra $I(k,k+1)$ to $F(A)$ for any $k\in\mathbb{N}$ because 
$F(A)$ is isomorphic to $M(A)^\infty\cap A^{\prime}/\mathrm{Ann}(A,M(A)^{\infty})$. 
\end{proof}

We denote by $\mathrm{Ped}(A)$ the Pedersen ideal of $A$. The Pedersen ideal $\mathrm{Ped}(A)$ is a minimal dense two-sided ideal of $A$. 
Hence every densely defined lower semicontinuous trace $\tau$ on $A$ is finite on $\mathrm{Ped}(A)$ because $\tau$ is finite on a dense two-sided ideal. 
Moreover for any positive element $h$ in $\mathrm{Ped}(A)$, $\overline{hAh}$ is contained in $\mathrm{Ped}(A)$. 
We refer the reader to \cite[II 5.2.4]{Bla} and \cite[Section 5.6]{Ped2} for details of the Pedersen ideal. 
If $A$ is unital, every densely defined lower semicontinuous trace on $A$ is bounded. 
Hence if $A$ is simple and $A\otimes\mathbb{K}$ has a nonzero projection, then 
there exists a full hereditary subalgebra $B$ of $A$ such that 
every densely defined lower semicontinuous trace on $B$ is bounded. 
In general, we have the following proposition. 
\begin{pro}\label{pro:induced trace}
Let $A$ be a $\sigma$-unital simple C$^*$-algebra. Then there exists 
a full hereditary subalgebra $B$ of $A$ such that 
every densely defined lower semicontinuous trace on $B$ is bounded. 
\end{pro}
\begin{proof}
Let $h$ be a nonzero positive element in $\mathrm{Ped}(A)$. 
Then any $\tau\in T(A)$ restricts to a bounded trace on $\overline{hAh}$ because every positive liner functional is automatically bounded. 
We obtain the conclusion by Corollary \ref{cor:folklore trace}. 
\end{proof}

If $A$ is separable, then $A$ is $\mathcal{Z}$-stable if and only if 
some full hereditary subalgebra is $\mathcal{Z}$-stable by Proposition \ref{pro:z-stable kirchberg} and Brown's theorem in \cite{B} 
since $F(A)$ is isomorphic to $F(A\otimes\mathbb{K})$. (See also \cite{TW1}.) 
Therefore we may assume that $A$ has no unbounded trace by the proposition above. 
Note that if $A$ has strict comparison and no unbounded trace, then for any 
$a,b\in A_{+}$ satisfying $d_{\tau}(a)< d_{\tau} (b)$ for all $\tau \in T_1(A)$, we have 
$a\precsim b$. 

\begin{pro}\label{pro:key trace}
Let $A$ be a separable C$^*$-algebra such that $T_1(A)$ is a non-empty compact set, and let $\{h_m\}_{m\in\mathbb{N}}$ be a countable 
approximate unit for $A$ and $\epsilon >0$. 
Then there exists a natural number $N$ such that 
$$
\max_{\tau\in T_1(A)}|\tau (f_n)-\tau(h_mf_n)| <\epsilon
$$
for any $m\geq N$ and for any sequence $(f_n)_{n\in\mathbb{N}}$ of  positive contractions in $A$. 
In particular, we have 
$$
\lim_{n\to\infty}\max_{\tau\in T_1(A)}| \tau (h_nf_n)-\tau (f_n) |=0.
$$ 
\end{pro}
\begin{proof}
For any $\tau\in T_1(A)$, we have $\tau(h_m)\leq \tau(h_{m+1})$ and $\lim\tau(h_m)=1$. 
By Dini's theorem, there exists a natural number $N$ such that 
$$
\max_{\tau\in T_1(A)}|1-\tau(h_m)| <\epsilon
$$
for any $m\geq N$. 
For any sequence $(f_n)_{n\in\mathbb{N}}$ of  positive contractions in $A$,
\begin{align*}
\max_{\tau\in T_1(A)}|\tau (f_n)-\tau(h_mf_n)| 
& =\max_{\tau\in T_1(A)}|\tau ((1-h_m)^{1/2}f_n(1-h_m)^{1/2})| \\
& \leq \max_{\tau\in T_1(A)}|1-\tau(h_m)| <\epsilon 
\end{align*}
for $m\geq N$. 
\end{proof}

We recall some definitions. 
\begin{Def}
Let $A$ be a separable C$^*$-algebra with no unbounded trace. Assume that $T_1(A)$ is a non-empty compact set. 
We say that $A$ has \textit{property (SI)} if for any central sequences $(e_n)_n$ and $(f_n)_n$ of positive contractions 
in $A$ satisfying 
$$
\lim_{n\to\infty}\max_{\tau\in T_1(A)}\tau (e_n)=0,\;\lim_{m\to\infty}\liminf_{n\to\infty}\min_{\tau\in T_1(A)}\tau (f_n^m)>0,
$$
there exists a central sequence $(s_n)_n$ in $A$ such that 
$$
\lim_{n\to\infty}\| s_n^*s_n-e_n\| =0, \; \lim_{n\to\infty}\| f_ns_n-s_n\| =0. 
$$
For a completely positive map $\varphi$ of $\tilde{A}$ to $\tilde{A}$, we say that $\varphi$ \textit{can be excised in small central sequences in $A$}
if for any central sequences $(e_n)_n$ and $(f_n)_n$ of positive contractions in $A$ satisfying the property above, there exists 
a sequence $(s_n)_{n\in\mathbb{N}}$ in $A$ such that 
$$
\lim_{n\to\infty}\| s_n^*as_n-\varphi (a)e_n\| =0\; \text{for any }a\in \tilde{A}, \; \lim_{n\to\infty}\| f_ns_n-s_n\| =0. 
$$
 
\end{Def}
\begin{rem}\label{rem:property si}
In the definition above, it is important that $e_n$ and $f_n$ are elements in $A$. 
We see that if $\mathrm{id}_{\tilde{A}}$ can be excised in small central sequences in $A$, then 
$A$ has property (SI) (see \cite[Proof of (iii)$\Rightarrow$(iv) of Theorem1.1]{MS}). 
\end{rem}

We shall generalize \cite[Lemma 4.6]{MS1} and \cite[Lemma 2.4]{MS} to non-unital C$^*$-algebras. 

\begin{lem}\label{lem:instead1}
Let $c$ be a positive element in a separable C$^*$-algebra $A$ such that $T_1(A)$ is a non-empty compact set, and 
let $\theta\in\mathbb{R}$. For any central sequence $(f_n)_n$ of positive contractions in $A$, we have 
$$
\limsup_{n\to\infty}\max_{\tau\in T_1(A)} |\tau (cf_n)-\theta\tau (f_n)|\leq 2\max_{\tau\in T_1(A)}|\tau (c)-\theta |.
$$
\end{lem}
\begin{proof}
Let $\{h_m\}_{m\in\mathbb{N}}$ be a countable approximate unit for $A$.
Replacing $f_n$ and $\theta$ in \cite[Lemma 4.6]{MS1} with $h_mf_n$ and $\theta h_m$ respectively, 
the same argument in the proof of \cite[Lemma 4.6]{MS1} shows that 
$$
\limsup_{n\to\infty}\max_{\tau\in T_1(A)} |\tau (cf_n)-\theta\tau (h_mf_n)|\leq 2\max_{\tau\in T_1(A)}|\tau (c)-\theta\tau (h_m) |
$$
for any $m\in\mathbb{N}$. 
By Proposition \ref{pro:key trace}, we have 
$$
\limsup_{n\to\infty}\max_{\tau\in T_1(A)} |\tau (cf_n)-\theta\tau (f_n)|\leq 2\max_{\tau\in T_1(A)}|\tau (c)-\theta |.
$$
\end{proof}
\begin{lem}\label{lem:instead2}
Let $A$ be a separable simple C$^*$-algebra such that $T_1(A)$ is a non-empty compact set, and let $a$ be a nonzero 
positive element in $\tilde{A}$. If $(f_n)_n$ is a central sequence of positive contractions in $A$ such that 
$$
\lim_{m\to\infty}\liminf_{n\to\infty}\min_{\tau\in T_1(A)} \tau (f_n^m)>0,
$$
then 
$$
\lim_{m\to\infty}\liminf_{n\to\infty}\min_{\tau\in T_1(A)} \tau (f_n^{m/2}af_n^{m/2})>0.
$$ 
\end{lem}
\begin{proof}
Put $R:=a^{1/2}A$. Since $A$ is simple, $R$ is a right ideal of $A$ such that $R^*R=AaA$ is a dense ideal of $A$. 
Therefore there exists a sequence $\{v_j\}_{j\in\mathbb{N}}$ in $A$ such that $\{\sum_{j=1}^nv_j^*av_j\}_{n\in\mathbb{N}}$ is an approximate unit for $A$ 
by a similar argument as in \cite[Lemma 2.3]{B}. 
By Proposition \ref{pro:key trace}, there exists a natural number $N$ such that 
$$
\lim_{m\to\infty}\liminf_{n\to\infty}\min_{\tau\in T_1(A)} \tau (\sum_{j=1}^Nv_j^*av_jf_n^m)>0.
$$
We have 
\begin{align*}
\lim_{m\to\infty}\liminf_{n\to\infty}\min_{\tau} \tau (\sum_{j=1}^Nv_j^*av_jf_n^m) 
& =\lim_{m\to\infty}\liminf_{n\to\infty}\min_{\tau}\sum_{j=1}^N\tau (v_j^*a^{1/2}f_n^ma^{1/2}v_j) \\
& =\lim_{m\to\infty}\liminf_{n\to\infty}\min_{\tau}\sum_{j=1}^N\tau (f_n^{m/2}a^{1/2}v_jv_j^*a^{1/2}f_n^{m/2}) \\
& \leq \sum_{j=1}^N\| v_j\|^2 \lim_{m\to\infty}\liminf_{n\to\infty}\min_{\tau} \tau (f_n^{m/2}af_n^{m/2}).
\end{align*}
Hence we obtain the conclusion. 
\end{proof}
Let $A$ be a separable simple C$^*$-algebra, and let $\tau$ be a tracial state on $A$. 
Consider the GNS representation $(\pi_{\tau},H_{\tau},\xi_{\tau})$ associated with $\tau$. 
Then $\pi_{\tau} (A){''}$ is a finite von Neumann algebra and  $\pi_{\tau} (A)$ is strongly dense subalgebra of  $\pi_{\tau} (A){''}$ in general. 
(Indeed, every approximate unit for $\pi_{\tau} (A)$ is strongly convergent to $1_{H_{\tau}}$.) 
In particular, Kaplansky density theorem shows that for any positive contraction $H\in\pi_{\tau} (A){''}$ there exists a sequence $\{a_n \}_{n\in\mathbb{N}}$ of positive contractions in $A$ 
such that $\pi (a_n)$ is strongly converge to $H$. 
We can identify C$^*(\pi_{\tau} (A), 1_{H_{\tau}})$ in $B(H_{\tau})$ with its unitization algebra $\tilde{A}$. 
Therefore the same proof as \cite[Lemma 2.1]{S} shows the following lemma. 
See also \cite[Proposition 3.5 and Theorem 4.3]{MS2}. 
\begin{lem}\label{lem:surjective central sequence}(\cite[Lemma 2.1]{S}) \\
Let $A$ be a separable simple nuclear C$^*$-algebra, and let $\tau$ be a tracial state on $A$. 
For any sequence $\{H_n\}_{n\in\mathbb{N}}$ of positive contractions in $\pi_{\tau} (A){''}$ such that 
$\| [H_n,x] \|_{\tau}\to 0$ for all $x\in\pi_{\tau} (A){''}$, there exists a central sequence $(c_n)_n$ of positive contractions in $A$ such that 
$\| c_n- H_n \|_{\tau}\to 0$. 
\end{lem}
Maybe someone considers that \cite[Lemma 2.1]{S} depends on a unit for the application of Haagerup's theorem (\cite[Theorem 3.1]{Haag}); see for example \cite[Theorem 2.1]{FKK} for details. 
But we can check that the same proof of \cite[Lemma 2.1]{S} works for non-unital C$^*$-algebras because $A$ is a two-sided ideal of $\tilde{A}$ 
and for any positive contraction $H\in\pi_{\tau} (A){''}$ there exists a sequence $\{a_n \}_{n\in\mathbb{N}}$ of positive contractions in $A$ 
such that $\pi (a_n)$ is strongly converge to $H$. 

If $\tau$ is an extremal tracial state on a separable simple infinite-dimensional nuclear C$^*$-algebra $A$, then $\pi_{\tau} (A){''}$ is the AFD II$_1$ factor in general. 
Therefore Lemma \ref{lem:surjective central sequence} and the same proof as \cite[Lemma 3.3]{MS} show the following lemma.  
\begin{lem}\label{lem:key MS}(\cite[Lemma 3.3]{MS}) \\
Let $A$ be a separable simple infinite-dimensional nuclear C$^*$-algebra with finitely many extremal tracial states. 
For any $k\in\mathbb{N}$, there exist central sequences $(c_{i,n})_n$ in $A$ , $i=1,2,..,k$ such that $c_{1,n}$ is a positive contraction 
for any $n\in\mathbb{N}$, $(c_{i,n}c_{j, n}^*)_n=\delta_{i.j}(c_{1,n}^2)_n$ and 
$$
\lim_{n\to\infty}\max_{\tau\in T_1(A)}| \tau (c_{1,n}^m)-\frac{1}{k} | =0 
$$
for any $m\in\mathbb{N}$. 
\end{lem}
Note that we need to consider unitaries in $\tilde{A}$ in the proof above. But it is also no problem because $A$ is a two-sided ideal of $\tilde{A}$. 

Let $\omega$ be a pure state on $A$. Then we can uniquely extend $\omega$ to a pure state $\tilde{\omega}$ on $\tilde{A}$. 
Moreover if $A$ is a separable simple non-type I C$^*$-algebra, then $\pi_{\omega} (A)\cap K(H_{\omega}) =\{0\}$. 
Therefore the same proof as \cite[Lemma 3.1]{MS} shows that every completely positive map of $\tilde{A}$ to $\tilde{A}$ can be approximated in the 
pointwise norm topology by completely positive map $\varphi$ of the form 
$$
\varphi (x)=\sum_{l=1}^N\sum_{i,j=1}^N\tilde{\omega}(d_i^*xd_j)c_{l,i}^*c_{l,j}, \;x\in \tilde{A}
$$
where $c_{l,i},d_i\in \tilde{A}$. 
For $1\leq l\leq N$, let $\varphi_l(x)=\sum_{i,j=1}^N\tilde{\omega}(d_i^*xd_j)c_{l,i}^*c_{l,j}$. Then $\varphi =\varphi_1+...+\varphi_N$. 
Using Lemma \ref{lem:instead2} instead of \cite[Lemma 2.4]{MS}, we can prove a version of \cite[Proposition 2.2]{MS}, i.e. that each 
$\varphi_l$ can be excised in small central sequences in $A$. (See the proof of \cite[Lemma 2.5]{MS}, which is where \cite[Lemma 2.4]{MS} gets used; note 
also this where we need strict comparison.) We can check that \cite[Lemma 3.4]{MS} holds without the assumption of a unit by using Lemma \ref{lem:instead1} and Lemma \ref{lem:key MS} instead of 
\cite[Lemma 4.6]{MS1} and \cite[Lemma 3.3]{MS} respectively. By this lemma, we see that a sum of completely positive maps $\tilde{A}\rightarrow\tilde{A}$, each 
of which can be excised in small central sequences in $A$, can itself be excised in small central sequences in $A$. 
Therefore we obtain the following theorem. 

\begin{thm}\label{thm:excise}
Let $A$ be a separable simple infinite-dimensional nuclear C$^*$-algebra with finitely many extremal tracial states and no unbounded trace. 
If $A$ has a strict comparison, then any completely positive map of $\tilde{A}$ to $\tilde{A}$ can be excised in small central sequences in $A$. 
\end{thm}

The following theorem is the main theorem in this section. 
\begin{thm}\label{thm:main stable}
Let $A$ be a separable simple infinite-dimensional non-type I nuclear C$^*$-algebra with 
a finite dimensional lattice of densely defined lower semicontinuous traces. 
Then $A$ has strict comparison if and only if $A$ is $\mathcal{Z}$-stable. 
\end{thm}
\begin{proof}
R\o rdam showed that if $A$ is $\mathcal{Z}$-stable, then $A$ has strict comparison (see \cite[Corollary 4.6]{Ror}). 
We shall show the only if part. 
By Proposition \ref{pro:induced trace}, we may assume that $A$ has no unbounded trace. Hence $A$ has property (SI) 
by Remark \ref{rem:property si} and Theorem \ref{thm:excise}. 
For any $k\in\mathbb{N}$, there exist central sequences $(c_{i,n})_n$ in $A$ , $i=1,2,..,k$ such that $c_{1,n}$ is a positive contraction 
, $(c_{i,n}c_{j, n}^*)_n=\delta_{i.j}(c_{1,n}^2)_n$ and 
$$
\lim_{n\to\infty}\max_{\tau\in T_1(A)}| \tau (c_{1,n}^m)-\frac{1}{k} | =0 
$$
for any $m\in\mathbb{N}$ by Lemma \ref{lem:key MS}. 
Let $\{h_n\}_{n\in\mathbb{N}}$ be an approximate unit for $A$. 
Taking a suitable subsequence of $\{h_n\}_{n\in\mathbb{N}}$, we may assume that $(h_n)_n(c_{1,n})_n= (c_{1,n})_n(h_n)_n$. 
Define central sequences $(f_{i,n})_n$ in $A$, $i=1,..,k$ by  
$(f_{i,n})_n:=(c_{i,n}h_n^{1/2})_n$, and put $(e_n)_n:=(h_n-\sum_{i=1}^kf_{i,n}^*f_{i,n})_n$. 
Then we may assume that $(e_n)_n$ is a central sequence of positive contractions in $A$. 
Proposition \ref{pro:key trace} implies $\lim_{n\to\infty}\max_{\tau}|\tau(f_{i,n}^*f_{i,n}-c_{i,n}^*c_{i,n})|=0$ for any 
$1\leq i\leq k$, and hence we have
$$
\lim_{n\to\infty}\max_{\tau\in T_1(A)}\tau (e_n)=0.
$$ 
Note that $(f_{1,n})_n$ is a central sequence of positive contractions in $A$ by the assumption of $(h_n)_n$.  
Because $\{h_n^{1/2}\}_{n\in\mathbb{N}}$ is also an approximate unit for $A$, we have
\begin{align*}
\limsup_{n\to\infty}\max_{\tau}\| c_{1,n}-c_{1,n}^{1/2}h_n^{1/2}c_{1,n}^{1/2}\|_{\tau}^2 
& = \limsup_{n\to\infty}\max_{\tau}\tau ((c_{1,n}-c_{1,n}^{1/2}h_n^{1/2}c_{1,n}^{1/2})^2) \\
& \leq \limsup_{n\to\infty}\max_{\tau}\tau (c_{1,n}-c_{1,n}^{1/2}h_n^{1/2}c_{1,n}^{1/2}) \\
& =0
\end{align*}
by Proposition \ref{pro:key trace}. 
Hence 
\begin{align*}
\limsup_{n\to\infty}\max_{\tau}|\tau (c_{1,n}^m)-\tau (f_{1,n}^m)|
& =\limsup_{n\to\infty}\max_{\tau}|\tau (c_{1,n}^m-(c_{1,n}h_n^{1/2})^m)| \\
& =\limsup_{n\to\infty}\max_{\tau}|\tau (c_{1,n}^m-(c_{1,n}^{1/2}h_n^{1/2}c_{1,n}^{1/2})^m)| \\
& \leq \limsup_{n\to\infty}\max_{\tau}\| c_{1,n}^m-(c_{1,n}^{1/2}h_n^{1/2}c_{1,n}^{1/2})^m)\|_{\tau}=0
\end{align*}
for any $m\in\mathbb{N}$. 
Therefore we have 
$$
\lim_{m\to\infty}\liminf_{n\to\infty}\min_{\tau\in T_1(A)}\tau (f_{1,n}^m)=1/k>0.
$$ 
Since $A$ has property (SI), there exists a central sequence $(s_n)_n$ in $A$ such that 
$(s_n^*s_n+\sum_{i=1}^kf_{i,n}^*f_{i,n})_n=(h_n)_n$ and $(f_{1,n}s_n)_n=(s_n)_n$. 
We have $[(f_{i,n}f_{j, n}^*)_n]=\delta_{i.j}[(f_{1,n}^2)_n]$ and $[(s_n^*s_n+\sum_{i=1}^kf_{i,n}^*f_{i,n})_n]=1$ in $F(A)$ 
because $[(h_n^{1/2})_n]$ is a unit in $F(A)$. 
It follows from \cite[Proposition 2.1]{Sa} that there exists a unital homomorphism of $I(k,k+1)$ to $F(A)$. 
Consequently $A$ is $\mathcal{Z}$-stable by Proposition \ref{pro:z-stable kirchberg}. 
\end{proof}
\begin{rem}
Let $A$ be a separable simple infinite-dimensional non-type I nuclear C$^*$-algebra with 
a finite dimensional lattice of densely defined lower semicontinuous traces, that has strict comparison. 
Since $A$ is $\mathcal{Z}$-stable by the theorem above, there exists a unital homomorphism of $\mathcal{Z}$ to $M(A)^\infty\cap A^{\prime}$. 
But we do not know that we could show this fact directly without using Kirchberg's central sequence algebras. 
Note that if $A$ is non-unital, then there exists no unital homomorphism of $\mathcal{Z}$ to $(\tilde{A})^\infty\cap A^{\prime}$ because 
$\tilde{A}$ is not $\mathcal{Z}$-stable. 
\end{rem}
The following corollary is an immediate consequence of the theorem above and Corollary \ref{cor:main}. 
\begin{cor}
Let $A$ be a separable simple nuclear stably projectionless C$^*$-algebra with a unique tracial state and no unbounded trace. 
Assume that $A$ has strict comparison. 
Then we have the following exact sequence: 
\[\begin{CD}
      {1} @>>> \mathrm{Out}(A) @>\rho_A>> 
\mathrm{Pic}(A) @>T>> \mathcal{F}(A)
 @>>> {1} \end{CD}. \] 
\end{cor}
We shall consider some examples. 
We refer the reader to \cite{Ti} for details of slow dimensional growth for nonunital C$^*$-algebras. 
Tikuisis showed that if $A$ is a simple separable approximately subhomogeneous C$^*$-algebra with slow dimension growth, then $\mathrm{Cu}(A)$ is almost unperforated in 
\cite[Corollary 5.9]{Ti}. 
The following immediate corollary of this result and Theorem \ref{thm:main stable} is suggested by the referee. 
\begin{cor}\label{cor:1-dimensional}
Let $A$ be a simple separable non-type I approximately subhomogeneous C$^*$-algebra with slow dimension growth and a finite dimensional lattice of densely defined lower semicontinuous traces. 
Then $A$ is $\mathcal{Z}$-stable. 
\end{cor}
We say that $A$ is a \textit{1-dimensional NCCW complex} if $A$ is a pullback C$^*$-algebra of the form 
\[\begin{CD}
      A @>\pi_2>>   E \\
             @VV\pi_1V      @VV\rho V  \\                 
      C([0,1])\otimes F  @>\delta_0\oplus\delta_1>> F \oplus F
\end{CD} \]
where $E$ and $F$ are finite-dimensional C$^*$-algebras and $\delta_i$ is the evaluation map at $i$. Razak's building block $A(n,m)$ 
is a 1-dimensional NCCW complex. The Cuntz semigroup of a 1-dimensional NCCW complex was computed in \cite{APS}. 
Every simple inductive limit C$^*$-algebras of 1-dimensional NCCW complexes is approximately subhomogeneous C$^*$-algebra with slow dimension growth. 
The following example is also suggested by the referee. 
\begin{ex}
For $n\geq 2$, let $\mathcal{O}_{n}$ denote the Cuntz algebra generated by $n$ isometries 
$S_1,...,S_n$. 
Given $\lambda_1,...,\lambda_n\in\mathbb{R}$, there exists by universality a one-parameter 
automorphism group $\alpha$ of $\mathcal{O}_n$ given by $\alpha_t (S_j)=e^{it\lambda_{j}}S_j$. 
Kishimoto and Kumjian showed that if $\lambda_{j}$ are all nonzero of the same sign and 
$\{\lambda_1,...,\lambda_n \}$ generates $\mathbb{R}$ as a closed subgroup, then 
$\mathcal{O}_n\rtimes \mathbb{R}$ is a simple stable projectionless C$^*$-algebra with unique (up to scalar multiple) densely defined 
lower semicontinuous trace in \cite{KK1} and \cite{KK2}. 
In particular, $\mathcal{O}_n\rtimes_{\alpha}\mathbb{R}$ has a one parameter trace scaling automorphism group. 
Dean showed that there exist many sets of numbers $\{\lambda_1,..,\lambda_n\}$ such that $\mathcal{O}_n\rtimes_{\alpha}\mathbb{R}$ can be 
expressed as an inductive limit C$^*$-algebra of 1-dimensional NCCW-complexes in \cite[Theorem 5.1]{Dean}. 
Therefore for $\alpha$ defined by such a set of numbers $\{\lambda_1,..,\lambda_n\}$, $\mathcal{O}_n\rtimes_{\alpha}\mathbb{R}$ is $\mathcal{Z}$-stable. 
Moreover it can be checked that for any positive element $h$ in $\mathrm{Ped}(\mathcal{O}_n\rtimes_{\alpha}\mathbb{R})$,  
$\mathrm{Pic}(\overline{h(\mathcal{O}_n\rtimes_{\alpha}\mathbb{R})h})$ is isomorphic to a semidirect product of 
$\mathrm{Out}(\overline{h(\mathcal{O}_n\rtimes_{\alpha}\mathbb{R})h})$ with $\mathbb{R}_{+}^\times$ by the same argument of the proof in Theorem \ref{thm:example main}. 
\end{ex}
Robert classified inductive limit C$^*$-algebras of 1-dimensional NCCW complexes with trivial K$_1$-groups in \cite{Rob}. 
\begin{cor}
Let $A$ be a simple stably projectionless C$^*$-algebra with a unique tracial state and no unbounded trace, that is expressible as an 
inductive limit C$^*$-algebra of 1-dimensional NCCW-complexes with trivial K$_1$-groups and 
$B$ a separable simple C$^*$-algebra with a unique tracial state and no unbounded trace. 
Then we have the following exact sequence: 
\[\begin{CD}
      {1} @>>> \mathrm{Out}(A\otimes B) @>\rho_{A\otimes B}>> 
\mathrm{Pic}(A\otimes B) @>T>> \mathbb{R}_{+}^\times
 @>>> {1} \end{CD}. \] 
\end{cor}
\begin{proof}
For any $r\in (0,1)$ there exists a positive element $h$ in $A$ such that $d_{\tau}(h)=r$ because $A$ has a positive element with a continuous spectrum. 
Note that the class of C$^*$-algebras covered by Robert's classification theorem in \cite{Rob} is closed under stable isomorphism (see \cite[Theorem 1.0.1]{Rob}). 
By \cite[Proposition 3.1.7]{Rob}, we see that a classifying invariant of the class of C$^*$-algebras which contains $A$ and $\overline{hAh}$ is equal to that of \cite[Corollary 6.2.4]{Rob}. 
Hence we see that $A$ is isomorphic to $\overline{hAh}$. 
Therefore $\mathcal{F}(A)=\mathbb{R}_{+}^\times$. Since $\mathcal{F}(A\otimes B)=\mathbb{R}_{+}^\times$ and $A\otimes B$ is separable, $A\otimes B$ is a stably projectionless C$^*$-algebra 
by \cite[Corollary 4.10]{Na}. 
Therefore we obtain the conclusion by Corollary \ref{cor:main} and Corollary \ref{cor:1-dimensional}. 
\end{proof}
We do not know whether the exact sequence above splits. 
This question is related to the existence of a one parameter trace scaling automorphism 
group of $A\otimes \mathbb{K}$. 
For any countable abelian groups $G_1$ and $G_2$, Kishimoto showed that there exists a stable projectionless simple separable nuclear 
C$^*$-algebra $A$ with unique (up to scalar multiple) densely defined lower semicontinuous 
trace with K$_0(A)=G_1$ and K$_1(A)=G_2$ in \cite{K2}. These stably projectionless C$^*$-algebras are constructed as 
the crossed products $\mathcal{O}\rtimes_{\alpha}\mathbb{R}$ by certain one parameter automorphism groups $\alpha$ of Kirchberg algebras $\mathcal{O}$ and 
the dual actions of $\alpha$ are trace scaling actions of $\mathcal{O}\rtimes_{\alpha}\mathbb{R}$. 
Hence it is natural to believe that there exists a kind of duality between $\mathcal{Z}$-stable stably projectionless 
C$^*$-algebras (with unique trace) and $\mathcal{O}_{\infty}$-stable C$^*$-algebras. 
From this view point, it seems to be possible to introduce the stably projectionless C$^*$-algebra $\mathcal{W}_n$ for any $n\geq 3$. 
Hence we denote by $\mathcal{W}_{2}$ the Razak-Jacelon algebra. 
On the other hand, Tikuisis \cite{Ti} constructed a simple separable nuclear stably 
projectionless C$^*$-algebra whose Cuntz semigroup is not almost unperforated. 

\section*{Acknowledgments}
The author would like to thank the Fields Institute, where a part of this work was done, 
for their hospitality. This travel is supported by the Global COE program 
"Education and Research Hub for Mathematics-for-Industry" at Kyushu University. 
He is also grateful to Leonel Robert for informing him about some results in \cite{Rob2} and 
to Hiroki Matui for useful suggestions. 
We also thank the referee for his or her careful reading and many valuable suggestions. 

\end{document}